\documentclass[a4paper,10pt]{amsart}

\usepackage[utf8]{inputenc}
\usepackage[ngerman,english]{babel}
\usepackage{array,tikz,listings,color,graphicx,csquotes,libertine,datetime,enumitem}
\usetikzlibrary{calc,tikzmark}
\usepackage[hypertexnames=false]{hyperref}

\numberwithin{equation}{section}
\numberwithin{table}{section}
\numberwithin{figure}{section}
\makeatletter
\@addtoreset{section}{part} 
\@addtoreset{equation}{part}
\@addtoreset{figure}{part}
\@addtoreset{table}{part}
\@addtoreset{algorithm}{part}
\@addtoreset{footnote}{section}
\makeatother
\newcommand\invisiblepart[1]{%
  \refstepcounter{part}%
  \addcontentsline{toc}{part}{#1}%
}
\makeatletter
\let\@cleartopmattertags\relax
\makeatother

\newif\ifen
\newcommand{\en}[1]{\ifen#1\fi} 
\newcommand{\de}[1]{\ifen\else#1\fi} 

\newtheorem{defi}{\en{Definition}\de{Definition}}[section]

\newtheorem{thm}[defi]{\en{Proposition}\de{Satz}}

\newtheorem{bem}[defi]{\en{Remark}\de{Bemerkung}}
\usepackage{mdframed}
 	\definecolor{lightergray}{rgb}{0.9, 0.9, 0.9}
 	\definecolor{darkergray}{rgb}{0.4, 0.4, 0.4}
\newmdtheoremenv[innertopmargin=0pt, backgroundcolor=lightergray, linecolor=gray, linewidth=1pt, topline=false, bottomline=false]{theo}[defi]{\en{Theorem}\de{Theorem}}

\usepackage{fancyhdr}
\makeatletter
  \def\section{\@startsection{section}{1}%
    \z@{.7\linespacing\@plus\linespacing}{.6\linespacing}%
    {\Large\normalfont\scshape\bfseries\centering}}
\makeatother

\newcommand{\ko}{\en{.}\de{{,}}} 
\newcommand{\AGM}{\operatorname{AGM}}
\newcommand\varpm{\mathbin{\vcenter{\hbox{
  \oalign{\hfil$\scriptstyle+$\hfil\cr
          \noalign{\kern-.3ex}
          $\scriptscriptstyle({-})$\cr}%
}}}}
\newcommand{\re}{\operatorname{Re}} 

\newcommand{\kommentInh}[2]{{\noindent\textbf{\ref{\en{EN}#1}.~~~\nameref{\en{EN}#1}~~\dotfill~~\pageref{\en{EN}#1}}\par\vspace*{1.5pt}
                            \noindent\narrower\narrower{\itshape #2 }\par\vspace*{7pt}}}

\def\lk{\mathopen{}\left(}
\def\rk{\right)\mathclose{}}

\usepackage{algorithm,algpseudocode}

\makeatletter
\renewcommand*{\ALG@name}{Algorithm\de{us}}
\makeatother

\begin{document}
  \selectlanguage{ngerman}

    \title{\Large{Easy Proof of Three Recursive $\pi$-Algorithms}\\[2ex]
        \hrule~\\[1.5ex]
        \Large{Einfacher Beweis dreier rekursiver $\pi$-Algorithmen}}
    \author{Lorenz Milla, \en{June}\de{Juni} 2025}

  \entrue
  \invisiblepart{English version}
  \fancyhead[OL]{\emph{\thesection.~\leftmark}}
\selectlanguage{ngerman}

\maketitle

\thispagestyle{empty}
\vspace*{2.25cm}
\label{\en{EN}titlepage}

{\narrower{
\noindent{\scshape Abstract.} This paper consists of three independent parts:

First we use only elementary algebra to prove that the quartic algorithm of the Borwein brothers has exactly the same output as the Brent-Salamin algorithm, but that the latter needs twice as many iterations.

Second we use integral calculus to prove that the Brent-Salamin algorithm approximates $\pi$.
Combining these results proves that the Borwein brothers' quartic algorithm also approximates $\pi$.

Third, we prove the quadratic convergence of the Brent-Salamin algorithm, which also proves the quartic convergence of Borwein's algorithm.\\
\vspace*{2mm}
{\begin{center}\emph{English version: pp.~\pageref{ENtitlepage}--\pageref{ENliterat}}\end{center}}

\vspace*{1.5cm}

\noindent{\scshape Zusammenfassung.} 
    Dieses Paper besteht aus drei unabhängigen Teilen:
    
    Erstens beweisen wir mit elementarer Algebra, dass der Borwein-Algorithmus vierter Ordnung die gleichen Ergebnisse liefert wie der Brent-Salamin-Algo"-rithmus, wobei letzterer doppelt so viele Iterationen benötigt.
    
    Zweitens beweisen wir mit Integralrechnung, dass der Brent-Salamin-Algo-\linebreak rithmus gegen $\pi$ konvergiert.
    Hieraus folgt, dass der Borwein-Algorithmus vierter Ordnung ebenfalls gegen $\pi$ konvergiert.
    
    Drittens beweisen wir die quadratische Konvergenz des Brent-Salamin-Algo-\linebreak rithmus und somit auch die quartische Konvergenz des Borwein-Algorithmus.\\
\vspace*{2mm}
{\begin{center}\emph{Deutsche Version: S.~\pageref{titlepage}--\pageref{literat}}\end{center}}

~
}}

\vfill\pagebreak\renewcommand{\proofname}{\en{Proof}\de{Beweis}}
\section*{\en{Introduction: The Algorithms}\de{Einleitung: Die Algorithmen}}
\label{\en{EN}kapeinl}
\renewcommand{\leftmark}{\en{Introduction: The Algorithms}\de{Einleitung: Die Algorithmen}}
\noindent \en{This paper is about the following three recursive $\pi$-algorithms:}%
\de{Dieses Paper handelt von den folgenden drei rekursiven $\pi$-Algorithmen:}

\begin{algorithm}[H]
\caption{\en{(Brent \& Salamin) (or: Gauß \& Legendre) (or: ``AGM Iteration'')}%
\de{(Brent \& Salamin) (oder: Gauß \& Legendre) (oder: \glqq AGM-Iteration\grqq)}}\label{\en{EN}algBS}
\begin{align*}
    &\left\{
    \begin{aligned}
        a_0 &:= 1\\ b_0 &:= \frac 1 {\sqrt{2}}
    \end{aligned}
    \right\}
    ~~ \text{\en{and}\de{und}} ~~
    \left\{
    \begin{aligned}
     a_{n} &:= \frac{a_{n-1}+b_{n-1}}{2} && \text{(\en{arithmetic mean}\de{arithmetisches Mittel})}\\
     b_{n} &:= \sqrt{a_{n-1}\cdot b_{n-1}} && \text{(\en{geometic mean}\de{geometrisches Mittel})}\\
     c_{n}^2 &:= a_{n}^2-b_{n}^2 && 
    \end{aligned}
    \right\}\\
    &\Longrightarrow~\text{\en{output after $N$ iterations:}\de{Ausgabe nach $N$ Iterationen:}} ~~~ p_N := \frac{(a_N+b_N)^2}{1-2\cdot\sum_{j=1}^N 2^j\cdot c_j^2}
\end{align*}
\end{algorithm}

\vspace*{-2mm}
\begin{algorithm}[H]
\caption{\en{(Borwein \& Borwein, quadratic convergence)}\de{(Borwein \& Borwein, quadratische Konvergenz)}}\label{\en{EN}algBB2}
\begin{align*}
    &\left\{
    \begin{aligned}
        k_{0} &:= 3-2\cdot{\sqrt {2}}\\ e_{0} &:= 6-4\cdot{\sqrt {2}}
    \end{aligned}
    \right\}
    ~~ \text{\en{and}\de{und}} ~~
    \left\{
    \begin{aligned}
     k_n&:=\frac{1-\sqrt{1-k_{n-1}^{2}}}{1+\sqrt{1-k_{n-1}^{2}}}\\
     e_n&:=e_{n-1}\cdot(1+k_{n})^{2}- 2^{n+1}\cdot k_n
    \end{aligned}
    \right\}\\
    &\Longrightarrow~\text{\en{output after $N$ iterations:}\de{Ausgabe nach $N$ Iterationen:}} ~~~ \widehat{\pi}_N := \frac{1}{e_N}
\end{align*}
\end{algorithm}

\vspace*{-2mm}
\begin{algorithm}[H]
\caption{\en{(Borwein \& Borwein, fourth order convergence)}\de{(Borwein \& Borwein, Konvergenz vierter Ordnung)}}\label{\en{EN}algBB4}
\begin{align*}
    &\left\{
    \begin{aligned}
        y_{0}&:={\sqrt {2}}-1\\z_{0}&:=6-4\cdot{\sqrt {2}}
    \end{aligned}
    \right\}
    ~~ \text{\en{and}\de{und}} ~~
    \left\{
    \begin{aligned}
     y_n&:=\frac{1-\sqrt[4]{1-y_{n-1}^{4}}}{1+\sqrt[4]{1-y_{n-1}^{4}}}\\
     z_n&:=z_{n-1}\cdot(1+y_{n})^{4}-2\cdot 4^n\cdot y_{n}\cdot(1+y_{n}+y_{n}^{2})
    \end{aligned}
    \right\}\\
    &\Longrightarrow~\text{\en{output after $N$ iterations:}\de{Ausgabe nach $N$ Iterationen:}} ~~~ \pi_N := \frac{1}{z_N}
\end{align*}
\end{algorithm}

\en{We prove that these three algorithms produce the same approximations of $\pi$, where the number of correct digits is being \emph{doubled} or \emph{quadrupled} with each iteration.}%
\de{Wir beweisen, dass diese drei Algorithmen die gleichen Näherungen der Zahl $\pi$ berechnen, wobei die Anzahl korrekter $\pi$-Dezimalen mit jeder Iteration ungefähr \emph{verdoppelt} bzw. \emph{vervierfacht} wird.}

\en{These results have been proven before, but we elaborate all intermediate calculations and we use only elementary algebra and integral calculus.}%
\de{Diese Resultate wurden bereits anderswo bewiesen, aber wir führen alle Rechnungen explizit aus und wir verwenden nur elementare Algebra und Integralrechnung.}

\en{Our proof consists of three independent chapters:}%
\de{Unser Beweis besteht aus drei unabhängigen Kapiteln:}\par\vspace*{8pt}

\kommentInh{kapborw}{\en{We prove that the algorithms have the same outputs: $\widehat{\pi}_N=p_N$ and $\pi_N=p_{2N}$.}%
\de{Wir beweisen, dass die drei Algorithmen die gleiche Ausgabe liefern, also dass $\widehat{\pi}_N=p_N$ und $\pi_N=p_{2N}$ gilt.}}

\kommentInh{gausskap}{\en{We prove that the output $p_N$ of Brent-Salamin converges to $\pi$ as $N\rightarrow\infty$.}%
\de{Wir beweisen, dass die Ausgabe $p_N$ des Brent-Salamin-Algorithmus gegen $\pi$ konvergiert.}}

\kommentInh{kapkonv}{\en{We prove the quadratic convergence of $p_N$:
~$|\pi-p_{n+1}| < 0\ko075 \cdot |\pi-p_{n}|^2$.\linebreak Here $\pi$ denotes the limit of $p_N$.}%
\de{Wir beweisen, dass $p_N$ quadratisch konvergiert: $|\pi-p_{n+1}| < 0\ko075 \cdot |\pi-p_{n}|^2$. Hierbei bezeichnet $\pi$ den Grenzwert von $p_N$.}}\vspace*{-7pt}

\vfill\pagebreak
\section{\en{Proof of Equivalence of the Algorithms}\de{Beweis der Äquivalenz der Algorithmen}}\label{\en{EN}kapborw}
\renewcommand{\leftmark}{\en{Proof of Equivalence of the Algorithms}\de{Beweis der Äquivalenz der Algorithmen}}
\en{We call two algorithms ``equivalent'' if they produce the same outputs. In this chapter we will prove that the three algorithms on p.~\pageref{\en{EN}algBB2} are equivalent. More precisely:}%
\de{Zwei Algorithmen, die exakt die gleichen Ergebnisse ausgeben, nennen wir äquivalent.
Wir werden beweisen, dass die drei Algorithmen von S. \pageref{\en{EN}algBS} äquivalent sind. Genauer:}\\

\begin{theo}\label{\en{EN}theoequi}
\en{For the outputs of the three algorithms on p.~\pageref{\en{EN}algBS}, where
\begin{itemize}[leftmargin=*]
\item $p_N$ is the output of the Brent-Salamin Alg.~\ref{\en{EN}algBS},
\item $\widehat{\pi}_N$ is the output of the Borweins' quadratic Alg.~\ref{\en{EN}algBB2},
\item $\pi_N$ is the output of the Borweins' fourth order Alg.~\ref{\en{EN}algBB4},
\end{itemize}
it holds:
$$\widehat{\pi}_N=p_N\qquad\text{\en{and}\de{und}}\qquad \pi_N=\widehat{\pi}_{2N}=p_{2N}$$
thus these algorithms produce the same sequence of outputs if the outputs are calculated exactly.}%
\de{Für die Ausgaben der drei Algorithmen auf S.~\pageref{\en{EN}algBS}, wobei
\begin{itemize}[leftmargin=*]
\item $p_N$ die Ausgabe des Brent-Salamin Alg.~\ref{\en{EN}algBS} bezeichne,
\item $\widehat{\pi}_N$ die Ausgabe des quadratischen Borwein-Alg.~\ref{\en{EN}algBB2} bezeichne und
\item $\pi_N$ die Ausgabe des quartischen Borwein-Alg.~\ref{\en{EN}algBB4} bezeichne,
\end{itemize}
gilt:
$$\widehat{\pi}_N=p_N\qquad\text{\en{and}\de{und}}\qquad \pi_N=\widehat{\pi}_{2N}=p_{2N}$$
d.h. dass diese Algorithmen genau die gleichen Ergebnisse liefern, falls die Ergebnisse exakt berechnet werden.}
\end{theo}
\begin{proof}\renewcommand{\qedsymbol}{}
\en{This has been proven by Brent \cite{\en{EN}Brent2018} who used elliptic modular functions and by Guillera \cite{\en{EN}Guillera} who used a theorem of Gauss, but we will need only elementary algebra for the proof of $\widehat{\pi}_N=p_N$ in Prop.~\ref{\en{EN}equi1} and for the proof of $\pi_N=\widehat{\pi}_{2N}$ in Prop.~\ref{\en{EN}equi2}.}%
\de{Das wurde bereits von Brent \cite{\en{EN}Brent2018} mit Hilfe elliptischer Modulfunktionen bewiesen und von Guillera \cite{\en{EN}Guillera}, der eine Formel von Gauss verwendet.
Wir benötigen nur elementare Algebra für den Beweis von $\widehat{\pi}_N=p_N$ in Satz \ref{\en{EN}equi1} und von $\pi_N=\widehat{\pi}_{2N}$ in Satz \ref{\en{EN}equi2}.}
\end{proof}

\begin{bem}
\en{When the computations are done using floating-point or interval arithmetic, the initial values and the iterations can only be done with finite precision. This produces rounding errors which propagate differently in the three algorithms. The outputs of the algorithms thus differ in the last decimals.
To compute $D$ decimals of $\pi$ correctly, one has to compute all initial and intermediate values to some extended precision (e.g. to $D+50$ decimals), and the additional decimals have to be cut off in the end.}%
\de{Bei einer tatsächlichen Implementierung der Algorithmen mit Hilfe von Gleit"-komma- oder Intervallarithmetik kann man die Startwerte und die Iterationen nur mit einer endlichen Genauigkeit berechnen. Hier entstehen Rundungsfehler, die sich bei den verschiedenen Algorithmen unterschiedlich fortpflanzen. Die Ausgaben der Algorithmen unterscheiden sich also in den letzten  Dezimalen.
Um $D$ Dezimalen von $\pi$ korrekt zu berechnen, muss man von Anfang an alle Zwischenergebnisse auf einige zusätzliche Dezimalen berechnen (z.B. auf $D+50$ Dezimalen), die man am Ende wieder abschneidet.}
\end{bem}

\begin{thm}\label{\en{EN}equi1}
\en{For the sequences defined in the Brent-Salamin Alg.~\ref{\en{EN}algBS} and the Borweins' Alg.~\ref{\en{EN}algBB2} on p.~\pageref{\en{EN}algBS} it holds $e_n=1/p_n$ and $k_n=a_n/a_{n+1}-1$.
In particular it holds $$\widehat{\pi}_N=p_N$$ thus these two algorithms produce the same sequence of outputs.}%
\de{Für die Größen des Brent-Salamin-Algorithmus \ref{\en{EN}algBS} und des Borwein-Alg.~\ref{\en{EN}algBB2} auf S.~\pageref{\en{EN}algBS} gilt $e_n=1/p_n$ und $k_n=a_n/a_{n+1}-1$.
Insbesondere gilt $$\widehat{\pi}_N=p_N$$ d.h. diese beiden Algorithmen liefern genau die gleichen Ergebnisse.}
\end{thm}
\begin{proof}
\en{We set $E_n:=1/p_n$ and $K_n:=a_n/a_{n+1}-1$ and prove by induction that it holds $E_n=e_n$ and $K_n=k_n$:}%
\de{Wir setzen $E_n:=1/p_n$ und $K_n:=a_n/a_{n+1}-1$ und beweisen dann per vollständiger Induktion, dass $E_n=e_n$ und $K_n=k_n$ gilt:}
\begin{itemize}[leftmargin=*]
    \item \en{First we prove $K_0=k_0$ and $E_0=e_0$:}\de{Für den Induktionsanfang beweisen wir $K_0=k_0$ und $E_0=e_0$:}
\begin{align*}
    K_0&:=\frac{a_0}{a_1}-1=\frac{1}{(1+1/\sqrt{2})/2}-1=\frac{4}{2+\sqrt{2}}-1=\frac{4\cdot(2-\sqrt{2})}{4-2}-1\\
    &=2\cdot(2-\sqrt{2})-1=3-2\cdot\sqrt{2}=k_0\\
    E_0&:=\frac{1}{p_0} = \frac{1-2\sum_{j=1}^{0} 2^j\cdot c_j^2}{(a_0+b_0)^2} = \frac{1}{\left(1+1/\sqrt{2}\right)^2} = \frac{1}{1+\sqrt{2} + 1/2}\\
    &= \frac{2}{3+2\cdot\sqrt{2}}=\frac{2\cdot(3-2\cdot\sqrt{2})}{9-4\cdot 2} = 2\cdot(3-2\cdot\sqrt{2}) = e_0
\end{align*}
\item \en{Now we prove $K_n=k_n$ and $E_n=e_n$ using the induction hypothesis (which states $K_{n-1}=k_{n-1}$ and $E_{n-1}=e_{n-1}$):}%
\de{Beweise jetzt unter Verwendung der Induktionsvoraussetzungen ($K_{n-1}=k_{n-1}$ und $E_{n-1}=e_{n-1}$), dass $K_n=k_n$ und $E_n=e_n$ gilt:}
\begin{align*}
    K_n &:=\frac{a_n}{a_{n+1}}-1 = \frac{a_n-(a_n+b_n)/2}{(a_n+b_n)/2}=\frac{a_n-b_n}{a_n+b_n}\nonumber\\
    \Longrightarrow\quad K_n^2 &=\left(\frac{a_n-b_n}{a_n+b_n}\right)^2 = \frac{(a_n+b_n)^2-4a_nb_n}{(a_n+b_n)^2} = 1 - \frac{b_{n+1}^2}{a_{n+1}^2}\nonumber\\
    \Longrightarrow\quad \sqrt{1-K_{n-1}^2} &= \sqrt{1-\left(1 - \frac{b_{n}^2}{a_{n}^2}\right)}=\sqrt{b_n^2/a_n^2}=b_n/a_n
\end{align*}
\en{Using the induction hypothesis this yields $K_n=k_n$:}%
\de{Hieraus folgt nun $K_n=k_n$ aufgrund der Induktionsvoraussetzung:}
$$K_n =\frac{a_n-b_n}{a_n+b_n} =\frac{1-b_n/a_n}{1+b_n/a_n} = \frac{1-\sqrt{1-K_{n-1}^2}}{1+\sqrt{1-K_{n-1}^2}} = \frac{1-\sqrt{1-k_{n-1}^2}}{1+\sqrt{1-k_{n-1}^2}} = k_n$$
\en{From}\de{Schließlich gilt} $p_n :=\frac{(a_n+b_n)^2}{1-2\cdot\sum_{j=1}^n 2^j c_j^2}= \frac{4\cdot a_{n+1}^2}{1-2\cdot\sum_{j=1}^n 2^j c_j^2}$ \en{we obtain}\de{und somit}
$$E_n := \frac{1}{p_n} = \frac{1-2\cdot\sum_{j=1}^n 2^j c_j^2}{4\cdot a_{n+1}^2}$$
\en{This yields}\de{Das liefert}
\begin{align*}
    a_{n+1}^2\cdot E_n -a_n^2\cdot E_{n-1} &= \left(\frac 1 4 -\frac 2 4 \cdot\sum_{j=1}^n 2^j c_j^2\right) - \left(\frac 1 4-\frac 2 4 \cdot\sum_{j=1}^{n-1} 2^j c_j^2\right)=-2^{n-1}\cdot c_n^2\\
    \Longrightarrow\quad E_n &= \frac{a_n^2}{a_{n+1}^2}\cdot E_{n-1} - 2^{n-1}\cdot \frac{c_n^2}{a_{n+1}^2}
\end{align*}
\en{Using}\de{Mit} $\frac{c_n^2}{a_{n+1}^2}=\frac{a_n^2-b_n^2}{(a_n^2+b_n)^2/4}=4\cdot\frac{a_n-b_n}{a_n+b_n}=4\cdot\left(\frac{2a_n}{a_n+b_n}-\frac{a_n+b_n}{a_n+b_n}\right)=4\cdot\left(\frac{a_n}{a_{n+1}}-1\right)$ \en{we get}\de{folgt}:
\begin{align*}
    E_n &= \left(\frac{a_n}{a_{n+1}}\right)^2\cdot E_{n-1} - 2^{n+1}\cdot\left(\frac{a_n}{a_{n+1}}-1\right)
\end{align*}
\en{Here we replace $a_n/a_{n+1}$ by $K_n+1$ and obtain}%
\de{Hier ersetzen wir $a_n/a_{n+1}$ durch $K_n+1$ und erhalten}
$$E_n = \left(K_n+1\right)^2\cdot E_{n-1} - 2^{n+1}\cdot K_n$$
\en{But we already proved $K_n=k_n$. Thus the induction hypothesis $E_{n-1}=e_{n-1}$ implies:}%
\de{Aber wir haben bereits $K_n=k_n$ bewiesen und nach Induktionsvoraussetzung gilt $E_{n-1}=e_{n-1}$:}
$$\Longrightarrow\quad E_n = \left(k_n+1\right)^2\cdot e_{n-1} - 2^{n+1}\cdot k_n$$
\en{Here we recognize the definition of $e_n$, thus we have proven $E_n=e_n$.}%
\de{Hier erkennen wir die Definition von $e_n$, also ist auch $E_n=e_n$ bewiesen.}
\end{itemize}
\en{This proves $\widehat{\pi}_N=1/e_N=p_N$ for all $N\in\mathbb N$, thus the two algorithms produce the same sequence of outputs.}%
\de{Somit haben wir für alle $N\geq 0$ bewiesen, dass $\widehat{\pi}_N=1/e_N=p_N$ gilt, d.h. dass die beiden Algorithmen genau die gleichen Ergebnisse liefern.}
\end{proof}

\begin{thm}\label{\en{EN}equi2}
\en{For the sequences defined in the Borweins' Alg.~\ref{\en{EN}algBB2} and \ref{\en{EN}algBB4} on p.~\pageref{\en{EN}algBB2} it holds $y_n=\sqrt{k_{2n}}$ and $z_n=e_{2n}$.
In particular it holds $$\pi_N=\widehat{\pi}_{2N}$$ thus one iteration of Alg.~\ref{\en{EN}algBB4} is equivalent to two iterations of Alg.~\ref{\en{EN}algBB2}.}%
\de{Für die Größen der Algorithmen \ref{\en{EN}algBB2} und \ref{\en{EN}algBB4} auf Seite \pageref{\en{EN}algBB2} gilt $y_n=\sqrt{k_{2n}}$ und $z_n=e_{2n}$.
Insbesondere gilt $$\pi_N=\widehat{\pi}_{2N}$$ d.h. eine Iteration des Algorithmus \ref{\en{EN}algBB4} entspricht genau zwei Iterationen des Algorithmus \ref{\en{EN}algBB2}.}
\end{thm}
\begin{proof}
\en{We set $Y_n:=\sqrt{k_{2n}}$ and $Z_n:=e_{2n}$ and prove by induction that it holds $Y_n=y_n$ and $Z_n=z_n$:}%
\de{Wir setzen $Y_n:=\sqrt{k_{2n}}$ und $Z_n:=e_{2n}$ und beweisen dann per vollständiger Induktion, dass $Y_n=y_n$ und $Z_n=z_n$ gilt:}
\begin{itemize}[leftmargin=*]
    \item \en{First we observe that $Z_0:=e_0=6-4\cdot\sqrt{2}=z_0$. Then it holds $Y_0:=\sqrt{k_0}=y_0$, because $y_0^2=\left(\sqrt{2}-1\right)^2=3-2\cdot\sqrt{2}=k_0$.}\de{Für den Induktionsanfang erkennen wir zunächst $Z_0:=e_0=6-4\cdot\sqrt{2}=z_0$. Außerdem gilt $y_0^2=\left(\sqrt{2}-1\right)^2=3-2\cdot\sqrt{2}=k_0$, also folgt $Y_0:=\sqrt{k_0}=y_0$.}
    \item \en{Now we prove $Y_n=y_n$ and $Z_n=z_n$ using the induction hypothesis (which states $Y_{n-1}=y_{n-1}$ and $Z_{n-1}=z_{n-1}$):\\
    From $k_{n}:=\frac{1-\sqrt{1-k_{n-1}^2}}{1+\sqrt{1-k_{n-1}^2}}$ we get $\sqrt{1-k_{n-1}^2}=\frac{1-k_{n}}{1+k_{n}}$ and thus}%
    \de{Beweise jetzt unter Verwendung der Induktionsvoraussetzungen ($Y_{n-1}=y_{n-1}$ und $Z_{n-1}=z_{n-1}$), dass $Y_n=y_n$ und $Z_n=z_n$ gilt:\\
    Aus $k_{n}:=\frac{1-\sqrt{1-k_{n-1}^2}}{1+\sqrt{1-k_{n-1}^2}}$ folgt $\sqrt{1-k_{n-1}^2}=\frac{1-k_{n}}{1+k_{n}}$ und somit}
    $$k_{n-1} = \sqrt{1-\left(\frac{1-k_{n}}{1+k_{n}}\right)^2}=\sqrt{\frac{(1+k_{n})^2-(1-k_{n})^2}{(1+k_{n})^2}}=\frac{2\cdot\sqrt{k_{n}}}{1+k_{n}}$$
    \en{Next, $Y_n:=\sqrt{k_{2n}}$ yields $k_{2n}=Y_n^2$ and $k_{2n-1}=\frac{2\cdot\sqrt{k_{2n}}}{1+k_{2n}}=\frac{2\cdot Y_n}{1+Y_n^2}$. This implies:}%
    \de{Aus $Y_n:=\sqrt{k_{2n}}$ folgt $k_{2n}=Y_n^2$ und $k_{2n-1}=\frac{2\cdot\sqrt{k_{2n}}}{1+k_{2n}}=\frac{2\cdot Y_n}{1+Y_n^2}$.
    Hieraus folgt:}
    \begin{align*}
        \frac{2\cdot Y_n}{1+Y_n^2} &= k_{2n-1} = \frac{1-\sqrt{1-k_{2n-2}^2}}{1+\sqrt{1-k_{2n-2}^2}} = \frac{1-\sqrt{1-Y_{n-1}^4}}{1+\sqrt{1-Y_{n-1}^4}}\\
        \Longrightarrow\quad \sqrt{1-Y_{n-1}^4} &= \frac{1-\frac{2\cdot Y_n}{1+Y_n^2}}{1+\frac{2\cdot Y_n}{1+Y_n^2}}
        = \frac{1+Y_n^2-2\cdot Y_n}{1+Y_n^2+2\cdot Y_n} = \frac{\left(1-Y_n\right)^2}{\left(1+Y_n\right)^2}\\
        \Longrightarrow\quad \sqrt[4]{1-Y_{n-1}^4} &= \frac{1-Y_n}{1+Y_n}\\
        \Longrightarrow\quad Y_n &= \frac{1-\sqrt[4]{1-Y_{n-1}^4}}{1+\sqrt[4]{1-Y_{n-1}^4}}
        = \frac{1-\sqrt[4]{1-y_{n-1}^4}}{1+\sqrt[4]{1-y_{n-1}^4}} = y_n
    \end{align*}
    \en{Thus we have proven $Y_n=y_n$ using the induction hypothesis $Y_{n-1}=y_{n-1}$ in the last step. It remains to prove $Z_n=z_n$:}%
    \de{Somit ist bewiesen, dass $Y_n=y_n$ ist, wobei wir im letzten Schritt die Induktionsvoraussetzung $Y_{n-1}=y_{n-1}$ benutzt haben. Wir müssen nun noch $Z_n=z_n$ beweisen:}

    \en{From the definition of $e_n$ in Alg.~\ref{\en{EN}algBB2} we obtain}%
    \de{Aus der Definition von $e_n$ in Algorithmus \ref{\en{EN}algBB2} folgt}
    \begin{align*}
        e_{2n} &= e_{2n-1}\cdot(1+k_{2n})^{2}- 2^{2n+1}\cdot k_{2n}\\
        \text{\en{and}\de{und}}\quad e_{2n-1} &= e_{2n-2}\cdot(1+k_{2n-1})^{2}- 2^{2n}\cdot k_{2n-1}
    \end{align*}
    \en{Putting this representation of $e_{2n-1}$ into the one of $e_{2n}$ yields:}%
    \de{Wenn wir diese Darstellung von $e_{2n-1}$ in die für $e_{2n}$ einsetzen erhalten wir:}
    \begin{align*}
        e_{2n} &= \left[e_{2n-2}\cdot(1+k_{2n-1})^{2}- 2^{2n}\cdot k_{2n-1}\right]\cdot(1+k_{2n})^{2}- 2^{2n+1}\cdot k_{2n}\\
        &= e_{2n-2}\cdot\left[(1+k_{2n-1})^{2}\cdot(1+k_{2n})^{2}\right]
           - 2^{2n}\cdot\left[k_{2n-1}\cdot(1+k_{2n})^{2}+ 2\cdot k_{2n}\right]
    \end{align*}
    \en{Using $k_{2n}=Y_n^2=y_n^2$ and $k_{2n-1}=\frac{2\cdot Y_n}{1+Y_n^2}=\frac{2\cdot y_n}{1+y_n^2}$ we obtain:}%
    \de{Mit $k_{2n}=Y_n^2=y_n^2$ und $k_{2n-1}=\frac{2\cdot Y_n}{1+Y_n^2}=\frac{2\cdot y_n}{1+y_n^2}$ folgt:}
    \begin{align*}
        e_{2n} &= e_{2n-2}\cdot\left[\left(1+\frac{2\cdot y_n}{1+y_n^2}\right)^{2}\cdot(1+y_n^2)^{2}\right]
           - 2^{2n}\cdot\left[\frac{2\cdot y_n}{1+y_n^2}\cdot(1+y_n^2)^{2}+ 2\cdot y_n^2\right]\\
           &=e_{2n-2}\cdot\left[\left(1+y_n^2+2\cdot y_n\right)^{2}\right]
           - 2^{2n}\cdot\left[2\cdot y_n\cdot(1+y_n^2)+ 2\cdot y_n^2\right]\\
           &=e_{2n-2}\cdot\left(1+y_n\right)^{4}
           - 2^{2n+1}\cdot y_n\cdot\left(1+y_n+y_n^2\right)
    \end{align*}
    \en{Here we use the induction hypothesis $z_{n-1}=Z_{n-1}=e_{2n-2}$:}%
    \de{Hier nutzen wir die Induktionsvoraussetzung $z_{n-1}=Z_{n-1}=e_{2n-2}$:}
    \begin{align*}
        Z_n := e_{2n} &= e_{2n-2}\cdot\left(1+y_n\right)^{4}
           - 2^{2n+1}\cdot y_n\cdot\left(1+y_n+y_n^2\right)\\
           &= z_{n-1}\cdot\left(1+y_n\right)^{4}
           - 2^{2n+1}\cdot y_n\cdot\left(1+y_n+y_n^2\right) = z_n
    \end{align*}
\end{itemize}
\en{Thus we have proven that $\pi_N=1/z_N=1/e_{2N}=\widehat{\pi}_{2N}$ holds for all $N\in\mathbb N$, thus that Alg.~\ref{\en{EN}algBB4} produces every second output of Alg.~\ref{\en{EN}algBB2}.}%
\de{Somit haben wir für alle $N\geq 0$ bewiesen, dass $\pi_N=1/z_N=1/e_{2N}=\widehat{\pi}_{2N}$ gilt, d.h. dass also Algorithmus \ref{\en{EN}algBB4} genau jedes zweite Ergebnis von Algorithmus \ref{\en{EN}algBB2} produziert.}
\end{proof}

\begin{proof}[\en{Proof of}\de{Beweis des} Thm.~\ref{\en{EN}theoequi}]
\en{In Prop.~\ref{\en{EN}equi1} we proved $\widehat{\pi}_N=p_N$ and in Prop.~\ref{\en{EN}equi2} we proved $\pi_N=\widehat{\pi}_{2N}$ -- thus both statements from Thm.~\ref{\en{EN}theoequi} are proven, and the algorithms are equivalent.}%
\de{In Satz \ref{\en{EN}equi1} haben wir $\widehat{\pi}_N=p_N$ bewiesen und in Satz \ref{\en{EN}equi2} haben wir $\pi_N=\widehat{\pi}_{2N}$ bewiesen -- somit sind beide Aussagen des Thm.~\ref{\en{EN}theoequi} bewiesen, und die drei Algorithmen sind äquivalent.}
\end{proof}

\begin{bem}
\en{The first outputs of the three equivalent algorithms are:}%
\de{Die ersten Ausgaben der drei äquivalenten Algorithmen sind:}
\begin{align*}
  \pi_0 = \widehat{\pi}_0 = p_0 &= \color{darkergray}{2\ko91421~35623~73095~04880~16887~24209~69807~85696~71875}\ldots\\
          \widehat{\pi}_1 = p_1 &= \underline{3\ko14}\color{darkergray}{057~92505~22168~24831~13312~68975~82331~17734~40237}\ldots\\
  \pi_1 = \widehat{\pi}_2 = p_2 &= \underline{3\ko14159~26}\color{darkergray}{462~13542~28214~93444~31982~69577~43144~37223}\ldots\\
          \widehat{\pi}_3 = p_3 &= \underline{3\ko14159~26535~89793~238}\color{darkergray}{27~95127~74801~86397~43812~25504}\ldots\\
  \pi_2 = \widehat{\pi}_4 = p_4 &= \underline{3\ko14159~26535~89793~23846~26433~83279~50288~41971}~\color{darkergray}{14678}\ldots
\end{align*}
\end{bem}

\vfill\pagebreak
\section{\en{Proof of the Brent-Salamin Algorithm}\de{Beweis des Brent-Salamin-Algorithmus}}\label{\en{EN}gausskap}
\renewcommand{\leftmark}{\en{Proof of the Brent-Salamin Algorithm}\de{Beweis des Brent-Salamin-Algorithmus}}
\en{In this chapter we prove that the Brent Salamin algorithm converges to $\pi$. This proof elaborates \cite{\en{EN}lord_1992} and uses only integral calculus like integration by parts or by substitution (also: two-dimensional substitution).}%
\de{In diesem Kapitel beweisen wir, dass der Brent-Salamin-Algorithmus gegen $\pi$ konvergiert. Der vorliegende Beweis arbeitet \cite{\en{EN}lord_1992} aus und setzt nur Integrationstechniken wie partielle Integration und Integration durch Substitution (auch zweidimensional -- also den Transformationssatz) voraus.}\\

\begin{theo} \label{\en{EN}tgauss}
\en{It holds the following formula due to Gauß (1809), Brent (1976) and Salamin (1976):}%
\de{Es gilt die Formel von Gauß (1809), Brent (1976) und Salamin (1976)}
$$\displaystyle\pi = \frac{4\cdot \AGM(1;1/\sqrt{2})^2}{1-2\cdot\sum_{j=1}^\infty 2^j\cdot c_j^2}$$
\en{Here, $\AGM(1;1/\sqrt{2})$ denotes the arithmetic-geometric mean (i.e. the common limit of $a_n$ and $b_n$ from the Brent-Salamin algorithm on p.~\pageref{\en{EN}algBS}). In particular, the sequence}%
\de{wobei $\AGM(1;1/\sqrt{2})$ das arithmetisch-geometrische Mittel (also den gemeinsamen Grenzwert der Folgen $a_n$ und $b_n$ des Brent-Salamin-Algorithmus auf S. \pageref{\en{EN}algBS}) bezeichnet.
Insbesondere konvergiert die Folge}
$$p_N := \frac{(a_N+b_N)^2}{1-2\cdot\sum_{j=1}^N 2^j\cdot c_j^2}$$
\en{of the Brent-Salamin algorithm on p.~\pageref{\en{EN}algBS} converges to $\pi$.}%
\de{des Brent-Salamin-Algorithmus auf S. \pageref{\en{EN}algBS} gegen $\pi$.}
\end{theo}
\begin{proof}\renewcommand{\qedsymbol}{}
\en{First we generalize initial values of the Brent-Salamin algorithm to}%
\de{Wir verallgemeinern den Brent-Salamin-Algorithmus zunächst auf die Startwerte}
$$a_0 := a\qquad\text{\en{and}\de{und}}\qquad b_0 := b\qquad\text{\en{with}\de{mit}}\qquad 0<b<a$$
\en{Later (from Prop.~\ref{\en{EN}satzwert} onwards) we will use $a=1$ and $b=1/\sqrt{2}$. On p.~\pageref{\en{EN}proofgaussi} we will continue the proof of Thm.~\ref{\en{EN}tgauss}, but first we proof some auxiliary propositions:}%
\de{Später (ab Satz \ref{\en{EN}satzwert}) werden wir $a=1$ und $b=1/\sqrt{2}$ setzen. Auf Seite \pageref{\en{EN}proofgaussi} wird der Beweis von Thm.~\ref{\en{EN}tgauss} fortgesetzt, zunächst beweisen wir einige Hilfssätze:}
\end{proof}

\begin{thm}\label{\en{EN}inequmean}
\en{The geometric mean $\sqrt{x\cdot y}$ and the arithmetic mean $\frac{x+y}{2}$ of two positive real numbers $x\neq y$ satisfy:}%
\de{Für das geometrische Mittel $\sqrt{x\cdot y}$ und das arithmetische Mittel $\frac{x+y}{2}$ zweier positiver reeller Zahlen $x\neq y$ gilt:}
$$\sqrt{x\cdot y} < \frac{x+y}{2}$$
\end{thm}
\begin{proof}
\en{From $x\neq y$ we deduce:}%
\de{Zunächst gilt (weil n.V. $x\neq y$ ist):}
\begin{align*}
    0< (x-y)^{2} =x^{2}-2xy+y^{2} =x^{2}+2xy+y^{2}-4xy=(x+y)^{2}-4xy
\end{align*}
\en{This yields $4xy < (x+y)^2$ and proves that the geometric mean $\sqrt{x\cdot y}$ is less than the arithmetic mean $\frac{x+y}{2}$.}%
\de{Hieraus folgt $4xy < (x+y)^2$ und somit, dass das geometrische Mittel $\sqrt{x\cdot y}$ kleiner als das arithmetische Mittel $\frac{x+y}{2}$ ist.}
\end{proof}

\begin{thm}\label{\en{EN}Thm1}
\en{The sequences $a_n$ and $b_n$ of the Brent-Salamin algorithm \ref{\en{EN}algBS} converge to a common limit which we call $\AGM(a,b)$.}%
\de{Die Folgen $a_n$ und $b_n$ des Brent-Salamin-Algorithmus \ref{\en{EN}algBS} konvergieren gegen einen gemeinsamen Grenzwert, den wir $\AGM(a,b)$ nennen.}
\en{The convergence of $a_n \searrow \AGM(a,b)$ and of $b_n \nearrow \AGM(a,b)$ is strictly monotonic and it holds $c_{n+1}^2<\frac 1 4 c_n^2$.}%
\de{Die Konvergenz von $a_n \searrow \AGM(a,b)$ und $b_n \nearrow \AGM(a,b)$ erfolgt streng monoton und es gilt $c_{n+1}^2<\frac 1 4 c_n^2$.}
\end{thm}
\begin{proof}
\en{Prop.~\ref{\en{EN}inequmean} tells that $b_n < a_n$ holds for all $n$. This implies the strict monotonicity of $b_{n+1}=\sqrt{a_n\cdot b_n} > \sqrt{b_n\cdot b_n} = b_n$ and $a_{n+1} = \frac{a_n+b_n}{2} < \frac{a_n+a_n}{2} = a_n$.
Both sequences are bounded by $b=b_0\leq b_n<a_n\leq a_0=a$ and thus convergent.
For $c_{n+1}^2$ it holds:}%
\de{Aus Satz \ref{\en{EN}inequmean} folgt, dass $b_n < a_n$ für alle $n$ gilt. Hieraus folgt die strenge Monotonie $b_{n+1}=\sqrt{a_n\cdot b_n} > \sqrt{b_n\cdot b_n} = b_n$ und $a_{n+1} = \frac{a_n+b_n}{2} < \frac{a_n+a_n}{2} = a_n$.\linebreak
Beide Folgen sind durch $b=b_0\leq b_n<a_n\leq a_0=a$ beschränkt und somit konvergent.
Für ihre Abweichung gilt:}
\begin{align}
    c_{n+1}^2 &=a_{n+1}^2-b_{n+1}^2= \left(\frac{a_n+b_n}{2}\right)^2-a_n\cdot b_n
    = \frac{a_n^2+2a_nb_n+b_n^2-4a_nb_n}{4}\nonumber\\
    &=\frac{(a_n-b_n)^2}{4} = \frac{a_n-b_n}{4(a_n+b_n)}\cdot(a_n^2-b_n^2) < \frac 1 4 \cdot c_n^2\label{\en{EN}glgcn}
\end{align}
\en{This proves that $c_n^2=a_n^2-b_n^2 < 4^{-n}\cdot c_0^2$ converges to zero and that $a_n$ and $b_n$ have the same limit.}%
\de{Hiermit ist bewiesen, dass $c_n^2=a_n^2-b_n^2 < 4^{-n}\cdot c_0^2$ gegen Null konvergiert, und dass $a_n$ und $b_n$ also gegen den selben Grenzwert konvergieren.}
\end{proof}

\pagebreak

\begin{thm}\label{\en{EN}thm2}
\en{The value of $$\displaystyle I(a,b) := \int\limits_0^{\pi/2}\frac{\mathrm{d}\Phi}{\sqrt{a^2\cos^2(\Phi)+b^2\sin^2(\Phi)}}$$
is constant on the whole AGM sequence, i.e. it holds $I(a_n,b_n)=I(a_0,b_0)$ for all $n\in\mathbb N$.}%
\de{Das Integral $$\displaystyle I(a,b) := \int\limits_0^{\pi/2}\frac{\mathrm{d}\Phi}{\sqrt{a^2\cos^2(\Phi)+b^2\sin^2(\Phi)}}$$
bleibt konstant über die ganze AGM-Folge, d.h. es gilt $I(a_n,b_n)=I(a_0,b_0)$ für alle $n\in\mathbb N$.}
\end{thm}
\begin{proof}
\en{First we substitute $t = b\cdot \tan \Phi$.
Then $1+\tan^2\Phi = \frac{1}{\cos^2\Phi}$ yields $\cos^2\Phi = \frac{b^2}{b^2+b^2\tan^2\Phi}=\frac{b^2}{b^2+t^2}$ and $\sin^2\Phi=1-\cos^2\Phi = \frac{b^2+t^2-b^2}{b^2+t^2} = \frac{t^2}{b^2+t^2}$.
Further, it holds $\frac{\mathrm{d}t}{\mathrm{d}\Phi} = b\cdot(1+\tan^2\Phi)=b+\frac{t^2}{b}=\frac{t^2+b^2}{b}$, thus $\frac{\mathrm{d}\Phi}{\mathrm{d}t} = \frac{b}{t^2+b^2}$. 
This shows that the substitution yields the following representation of $I(a,b)$:}%
\de{Wir führen zunächst die Substitution $t = b\cdot \tan \Phi$ durch:
Aus $1+\tan^2\Phi = \frac{1}{\cos^2\Phi}$ folgt dann $\cos^2\Phi = \frac{b^2}{b^2+b^2\tan^2\Phi}=\frac{b^2}{b^2+t^2}$ und $\sin^2\Phi=1-\cos^2\Phi = \frac{b^2+t^2-b^2}{b^2+t^2} = \frac{t^2}{b^2+t^2}$.
Außerdem erhalten wir $\frac{\mathrm{d}t}{\mathrm{d}\Phi} = b\cdot(1+\tan^2\Phi)=b+\frac{t^2}{b}=\frac{t^2+b^2}{b}$, also $\frac{\mathrm{d}\Phi}{\mathrm{d}t} = \frac{b}{t^2+b^2}$. Die genannte Substitution liefert also folgende alternative Darstellung von $I(a,b)$:}
\begin{align}
    I(a,b)&= \int\limits_0^{\infty} \frac{1}{\sqrt{a^2\cdot\frac{b^2}{b^2+t^2} + b^2\cdot\frac{t^2}{b^2+t^2}}}\cdot\frac{b}{t^2+b^2} \mathrm{d}t\nonumber\\
    &= \int\limits_0^{\infty} \frac{1}{\sqrt{(a^2+t^2)\cdot\frac{b^2}{b^2+t^2}}}\cdot\frac{b}{t^2+b^2} \mathrm{d}t\nonumber\\
    &= \int\limits_0^{\infty} \frac{\mathrm{d}t}{\sqrt{(t^2+a^2)\cdot(t^2+b^2)}}\label{\en{EN}glg22}
\end{align}
\en{Now we substitute $x=\frac{1}{2}\left(t-\frac{ab}{t}\right)$.
This yields $2xt = t^2 - ab$ and $t = x \varpm \sqrt{x^2+ab}$ (since $t>0$).
Thus it holds $\frac{\mathrm{d}t}{\mathrm{d}x} = 1 + \frac{2x}{2\sqrt{x^2+ab}} = \frac{\sqrt{x^2+ab}+x}{\sqrt{x^2+ab}}=\frac{t}{\sqrt{x^2+ab}}$ and:}%
\de{Nun substituieren wir $x=\frac{1}{2}\left(t-\frac{ab}{t}\right)$.
Das führt auf $2xt = t^2 - ab$ und (weil $t>0$) auf $t = x \varpm \sqrt{x^2+ab}$.
Somit gilt $\frac{\mathrm{d}t}{\mathrm{d}x} = 1 + \frac{2x}{2\sqrt{x^2+ab}} = \frac{\sqrt{x^2+ab}+x}{\sqrt{x^2+ab}}=\frac{t}{\sqrt{x^2+ab}}$, also:}
\begin{align*}
    I(a,b) &= \int\limits_{-\infty}^{\infty} \frac{1}{\sqrt{(t^2+a^2)\cdot(t^2+b^2)}}\cdot\frac{t}{\sqrt{x^2+ab}} \mathrm{d}x
    = \int\limits_{-\infty}^{\infty} \frac{\mathrm{d}x}{\sqrt{f(x)\cdot(x^2+ab)}}
\end{align*}
\en{Here we have denoted $\frac{(t^2+a^2)\cdot(t^2+b^2)}{t^2}$ by $f(x)$ (remember $t>0$). About $f(x)$ it holds:}%
\de{Hier haben wir $\frac{(t^2+a^2)\cdot(t^2+b^2)}{t^2}$ zu $f(x)$ zusammengefasst (beachte $t>0$). Für $f(x)$ gilt:}
\begin{align*}
f(x) & := \frac{(t^2+a^2)\cdot(t^2+b^2)}{t^2}= \frac{t^4+a^2t^2+b^2t^2+a^2b^2}{t^2}\\
&= t^2+\frac{a^2b^2}{t^2}+a^2+b^2 = \left(t-\frac{ab}{t}\right)^2+2ab+a^2+b^2 \\
&=(2x)^2+(a+b)^2
\end{align*}
\en{This yields:}%
\de{Für $I(a,b)$ erhalten wir also:}
\begin{align*}
    I(a,b) & = \int\limits_{-\infty}^{\infty} \frac{\mathrm{d}x}{\sqrt{((2x)^2+(a+b)^2)\cdot(x^2+ab)}}= \frac 1 2 \int\limits_{-\infty}^{\infty} \frac{\mathrm{d}x}{\sqrt{\left(x^2+\left(\frac{a+b}{2}\right)^2\right)\cdot(x^2+ab)}}
\end{align*}
\en{Here we use the fact that the integrand is even, thus $\frac 1 2 \int\limits_{-\infty}^{\infty}$ yields $\int\limits_{0}^{\infty}$:}%
\de{Hier nutzen wir, dass der Integrand eine gerade Funktion ist, weshalb $\frac 1 2 \int\limits_{-\infty}^{\infty}$ in $\int\limits_{0}^{\infty}$ übergeht:}
\begin{align*}
I(a,b) &= \int\limits_0^{\infty} \frac{\mathrm{d}x}{\sqrt{\left(x^2+\left(\frac{a+b}{2}\right)^2\right)\cdot(x^2+ab)}} = I\lk \frac{a+b}{2},\sqrt{ab}\rk
\end{align*}
\en{Now we have proven that for any $a>b>0$ it holds $I\lk \frac{a+b}{2},\sqrt{ab}\rk=I(a,b)$.
By induction this yields $I(a_n,b_n)=I(a_0,b_0)$ for all $n\in\mathbb N$.}%
\de{Wir haben also für beliebige $a>b>0$ bewiesen, dass $I\lk \frac{a+b}{2},\sqrt{ab}\rk=I(a,b)$ gilt.
Per vollständiger Induktion folgt hieraus $I(a_n,b_n)=I(a_0,b_0)$ für alle $n\in\mathbb N$.}
\end{proof}

\begin{thm}\label{\en{EN}satziagm}
\en{Let $I(a,b)$ be the integral from Prop.~\ref{\en{EN}thm2}. Then it holds:}%
\de{Für das Integral $I(a,b)$ aus Satz \ref{\en{EN}thm2} gilt:}
$$I(a,b) = \frac{\pi}{2\cdot\AGM(a,b)}$$
\end{thm}
\begin{proof}
\en{With $m:=\AGM(a,b)$, Prop.~\ref{\en{EN}Thm1} tells that $a_n$ and $b_n$ converge to $m$.
If we interchange the limit and the integration, Prop.~\ref{\en{EN}thm2} yields:}%
\de{Mit $m:=\AGM(a,b)$ gilt nach Satz \ref{\en{EN}Thm1}, dass $a_n\rightarrow m$ und $b_n \rightarrow m$ konvergieren.
Vertauschen von Grenzwertbildung und Integration liefert dann mit Satz \ref{\en{EN}thm2}:}\belowdisplayskip=-12pt
\begin{align*}
    I(a,b) &= I(a_n,b_n) = \lim_{n\rightarrow\infty} I(a_n,b_n) = I\left(\lim_{n\rightarrow\infty} a_n, \lim_{n\rightarrow\infty} b_n\right)= I(m,m)\\
    &=\int\limits_0^{\pi/2}\frac{\mathrm{d}\Phi}{\sqrt{m^2\cos^2(\Phi)+m^2\sin^2(\Phi)}} = \frac{\pi}{2}\cdot\frac{1}{m} = \frac{\pi}{2\cdot\AGM(a,b)}
\end{align*}
\end{proof}

\begin{thm}\label{\en{EN}satzLI}
\en{If we denote}%
\de{Für das Integral} $$\displaystyle L(a,b) := \int\limits_0^{\pi/2}\frac{\cos^2(\Phi)\mathrm{d}\Phi}{\sqrt{a^2\cos^2(\Phi)+b^2\sin^2(\Phi)}}$$
\en{then it holds}\de{gilt} $L(b,a)+L(a,b)=I(a,b)$ \en{and}\de{und} $L(b,a)-L(a,b) = \frac{a-b}{a+b}\cdot L(b_1,a_1)$.
\end{thm}
\begin{proof}
\en{To prove the first equation, we substitute $\Phi' = \frac{\pi}{2}-\Phi$. Then it holds $\cos(\Phi')=\sin(\Phi)$ and $\sin(\Phi')=\cos(\Phi)$, thus}%
\de{Um die erste Gleichung zu beweisen, substituieren wir $\Phi' = \frac{\pi}{2}-\Phi$. Dann gilt $\cos(\Phi')=\sin(\Phi)$ und $\sin(\Phi')=\cos(\Phi)$, also}
\begin{align*}
    L(b,a)&:= \int\limits_0^{\pi/2}\frac{\cos^2(\Phi)\mathrm{d}\Phi}{\sqrt{b^2\cos^2(\Phi)+a^2\sin^2(\Phi)}}
    = \int\limits_0^{\pi/2} \frac{\sin^2(\Phi')\mathrm{d}\Phi'}{\sqrt{b^2\sin^2(\Phi')+a^2\cos^2(\Phi')}}
\end{align*}
\en{With $\sin^2+\cos^2=1$ we deduce the first equation:}%
\de{Schließlich folgt aus $\sin^2+\cos^2=1$, dass gilt:}
$$L(b,a)+L(a,b) = \int\limits_0^{\pi/2}\frac{\left(\sin^2(\Phi)+\cos^2(\Phi)\right)\mathrm{d}\Phi}{\sqrt{a^2\cos^2(\Phi)+b^2\sin^2(\Phi)}} = I(a,b)$$
\en{Next we prove an alternative representation of $L(a,b)$, similar to the one of $I(a,b)$ in eq.~(\ref{\en{EN}glg22}): again we substitute $t = b\cdot \tan \Phi$ and obtain}%
\de{Als Nächstes beweisen wir eine alternative Darstellung von $L(a,b)$.
Genau wie bei $I(a,b)$ in Gleichung (\ref{\en{EN}glg22}) substituieren wir hierfür $t = b\cdot \tan \Phi$ und erhalten:}
\begin{align}
    L(a,b) = \int\limits_0^{\infty} \frac{\frac{b^2}{b^2+t^2}}{\sqrt{a^2\cdot\frac{b^2}{b^2+t^2} + b^2\cdot\frac{t^2}{b^2+t^2}}}\cdot\frac{b}{t^2+b^2} \mathrm{d}t
    =\int\limits_0^{\infty} \frac{\frac{b^2}{b^2+t^2} \mathrm{d}t}{\sqrt{(t^2+a^2)\cdot(t^2+b^2)}}\label{\en{EN}glg23}
\end{align}
\en{Then we calculate $L(b,a)$ by interchanging $a$ and $b$:}%
\de{Dann bilden wir die gesuchte Differenz $L(b,a)-L(a,b)$, wobei für den Ausdruck $L(b,a)$ die Variablen $a$ und $b$ vertauscht werden:}
\begin{align*}
    L(b,a)-L(a,b)&=\int\limits_0^{\infty}\frac{\frac{a^2}{a^2+t^2}-\frac{b^2}{b^2+t^2}}{\sqrt{(t^2+a^2)\cdot(t^2+b^2)}}\mathrm{d}t
\end{align*}
\en{From}\de{Hier gilt} $\frac{a^2}{a^2+t^2}-\frac{b^2}{b^2+t^2} = \frac{a^2(b^2+t^2)-b^2(a^2+t^2)}{(a^2+t^2)(b^2+t^2)} =\frac{a^2t^2-b^2t^2}{(a^2+t^2)(b^2+t^2)} =\frac{(a^2-b^2)t^2}{(a^2+t^2)(b^2+t^2)}$\en{ we deduce:}\de{, also:}
\begin{align*}
    L(b,a)-L(a,b)&= \int\limits_0^{\infty}\frac{\frac{(a^2-b^2)t^2}{(a^2+t^2)(b^2+t^2)}}{\sqrt{(t^2+a^2)\cdot(t^2+b^2)}}\mathrm{d}t
    = \int\limits_0^{\infty}\frac{(a^2-b^2)\cdot t^2}{(t^2+a^2)^{3/2}\cdot(t^2+b^2)^{3/2}}\mathrm{d}t
\end{align*}

\pagebreak\noindent
\en{And, as with $I(a,b)$, we substitute}%
\de{Und genau wie bei $I(a,b)$ substituieren wir nun} $x=\frac{1}{2}\left(t-\frac{ab}{t}\right)$:
\begin{align*}  
    L(b,a)-L(a,b)&= \int\limits_{-\infty}^{\infty} \frac{(a^2-b^2)t^2}{(t^2+a^2)^{3/2}\cdot(t^2+b^2)^{3/2}}\cdot\frac{t}{\sqrt{x^2+ab}} \mathrm{d}x\\
    &\tikzmark{Glg15}= \int\limits_{-\infty}^{\infty} \frac{(a^2-b^2)\mathrm{d}x}{f(x)^{3/2}\cdot\sqrt{x^2+ab}}\qquad\text{mit } f(x) := \frac{(t^2+a^2)\cdot(t^2+b^2)}{t^2},\nonumber
\end{align*}
\en{where again it holds $f(x) := \frac{(t^2+a^2)\cdot(t^2+b^2)}{t^2} = (2x)^2+(a+b)^2$ and thus:}%
\de{wobei wieder $f(x) := \frac{(t^2+a^2)\cdot(t^2+b^2)}{t^2} = (2x)^2+(a+b)^2$ gilt, also:}
\begin{align*}
   L(b,a)-L(a,b)&= \int\limits_{-\infty}^{\infty} \frac{(a^2-b^2)\mathrm{d}x}{((2x)^2+(a+b)^2)^{3/2}\cdot\sqrt{x^2+ab}}\\
    &= \frac{a^2-b^2}{8}\int\limits_{-\infty}^{\infty} \frac{\mathrm{d}x}{(x^2+a_1^2)^{3/2}\cdot\sqrt{x^2+b_1^2}}\\
    &= \frac{a^2-b^2}{8a_1^2}\cdot 2\int\limits_{0}^{\infty} \frac{\frac{a_1^2}{a_1^2+x^2}\mathrm{d}x}{\sqrt{(x^2+a_1^2)\cdot(x^2+b_1^2)}}
\end{align*}

\noindent \en{In this integral we recognize the representation (\ref{\en{EN}glg23}) of $L(b_1,a_1)$. This yields:}%
\de{In diesem Integral erkennen wir die Darstellung (\ref{\en{EN}glg23}) von $L(b_1,a_1)$. Es folgt:}
\begin{align*}
    L(b,a)-L(a,b) &= \frac{a^2-b^2}{4a_1^2}\cdot L(b_1,a_1) = \frac{(a-b)(a+b)}{(a+b)^2}\cdot L(b_1,a_1) \\
    &= \frac{a-b}{a+b}\cdot L(b_1,a_1)
\end{align*}
\en{Thus we have also proven the second equation.}%
\de{Somit ist auch die zweite Gleichung bewiesen.}
\end{proof}

\begin{thm}\label{\en{EN}sumthm}
\en{Denoting $S:=\sum_{j=1}^\infty 2^j\cdot c_j^2$ it holds:}%
\de{Mit $S:=\sum_{j=1}^\infty 2^j\cdot c_j^2$ gilt:}
$$2\cdot c_0^2\cdot L(a,b) = (c_0^2-S)\cdot I(a,b)$$
\end{thm}
\begin{proof}
\en{From $4\cdot(a_1^2-b_1^2)=4\cdot\left(\frac{a+b}{2}\right)^2-4ab=(a+b)^2-4ab=(a-b)^2$ we deduce (using both equations of Prop.~\ref{\en{EN}satzLI}):}%
\de{Zunächst gilt $4\cdot(a_1^2-b_1^2)=4\cdot\left(\frac{a+b}{2}\right)^2-4ab=(a+b)^2-4ab=(a-b)^2$.
Daraus folgt (unter Nutzung beider Gleichungen aus Satz \ref{\en{EN}satzLI}):}
\begin{align*}
    4\cdot(a_1^2-b_1^2)\cdot L(b_1,a_1) &= (a-b)^2\cdot L(b_1,a_1)\\
    &=(a^2-b^2)\cdot\frac{a-b}{a+b}\cdot L(b_1,a_1)\\
    &=(a^2-b^2)\cdot\left(L(b,a)-L(a,b)\right)\\
   &= (a^2-b^2)\cdot\left(L(b,a)-(I(a,b)-L(b,a))\right)\\
&= (a^2-b^2)\cdot\left(2\cdot L(b,a)-I(a,b)\right)
\end{align*}
\en{With the definition of $c_n$ this reads:}%
\de{Mit der Definition der $c_n$ können wir das wie folgt abkürzen:}
\begin{align*}
    4 \cdot c_1^2\cdot L(b_1,a_1) &= c_0^2\cdot\left(2\cdot L(b,a)-I(a,b)\right)\\
    \Longrightarrow\quad 2\cdot c_0^2 \cdot L(b,a) - 4\cdot c_1^2 \cdot L(b_1,a_1) &= c_0^2\cdot I(a,b)
\end{align*}
\en{Thus it holds for all}\de{Also gilt für alle} $j\in\mathbb N$:
\begin{align*}
    2\cdot c_j^2 \cdot L(b_j,a_j) - 4\cdot c_{j+1}^2 \cdot L(b_{j+1},a_{j+1}) &= c_j^2\cdot I(a_j,b_j)
\end{align*}
\en{Here we multiply with $2^{j}$ and use $I(a_j,b_j)=I(a,b)$ from Prop.~\ref{\en{EN}thm2}:}%
\de{Hier multiplizieren wir noch mit $2^{j}$ und nutzen $I(a_j,b_j)=I(a,b)$ aus Satz \ref{\en{EN}thm2}:}
\begin{align*}
    2^{j+1}\cdot c_j^2 \cdot L(b_j,a_j) - 2^{j+2}\cdot c_{j+1}^2 \cdot L(b_{j+1},a_{j+1}) &= 2^j\cdot c_j^2\cdot I(a,b)
\end{align*}
\en{Now we add these equations for $0\leq j\leq n$ and obtain:}%
\de{Nun summieren wir diese Gleichungen für $0\leq j\leq n$ und erhalten:}
\begin{align}
    \sum_{j=0}^n 2^{j+1}\cdot c_j^2 \cdot L(b_j,a_j) - \sum_{j=0}^n 2^{j+2}\cdot c_{j+1}^2 \cdot L(b_{j+1},a_{j+1}) &= \sum_{j=0}^n 2^j\cdot c_j^2\cdot I(a,b)\label{\en{EN}teleskop}
\end{align}
\en{In the second sum we shift the index $k=j+1$:}%
\de{In der zweiten Summe führen wir einen Indexshift $k=j+1$ durch:}
\begin{align*}
    \sum_{j=0}^n 2^{j+2}\cdot c_{j+1}^2 \cdot L(b_{j+1},a_{j+1}) &= \sum_{k=1}^{n+1} 2^{k+1}\cdot c_{k}^2 \cdot L(b_{k},a_{k})
\end{align*}
\en{This shows that the left side of (\ref{\en{EN}teleskop}) is a telescoping series, thus nearly all terms cancel each other out:}%
\de{Somit erkennen wir, dass auf der linken Seite von (\ref{\en{EN}teleskop}) eine Teleskopsumme steht, in der sich fast alle Summanden gegenseitig auslöschen:}
\begin{align}
    2^{0+1}\cdot c_0^2 \cdot L(b_0,a_0) - 2^{n+2}\cdot c_{n+1}^2 \cdot L(b_{n+1},a_{n+1}) &= \sum_{j=0}^n 2^j\cdot c_j^2\cdot I(a,b)\label{\en{EN}telesk}
\end{align}
\en{Here we estimate $L(b_{n+1},a_{n+1})<I(b_{n+1},a_{n+1})=I(b,a)$ and use $c_{n+1}^2<4^{-n-1}c_0^2$ from Prop.~\ref{\en{EN}Thm1}:}%
\de{Hier schätzen wir noch $L(b_{n+1},a_{n+1})<I(b_{n+1},a_{n+1})=I(b,a)$ und nutzen $c_{n+1}^2<4^{-n-1}c_0^2$ aus Satz \ref{\en{EN}Thm1}:}
\begin{align*}
    2^{n+2}\cdot c_{n+1}^2 \cdot L(b_{n+1},a_{n+1}) < 2^{n+2}\cdot 4^{-n-1}\cdot c_0^2\cdot I(b,a) = 2^{-n}\cdot c_0^2\cdot I(b,a)
\end{align*}
\en{Thus the second term of (\ref{\en{EN}telesk}) tends to $0$ (for $n\rightarrow\infty$) and we obtain:}%
\de{Also geht der zweite Ausdruck aus (\ref{\en{EN}telesk}) gegen Null (für $n\rightarrow\infty$) und wir erhalten:}
\begin{align*}
    2^{0+1}\cdot c_0^2 \cdot L(b_0,a_0) &= \sum_{j=0}^\infty 2^j\cdot c_j^2\cdot I(a,b)\\
    \Longrightarrow\quad 2 c_0^2 \cdot L(b,a) &= (c_0^2+S)\cdot I(a,b)
\end{align*}
\en{Finally we use $L(b,a)=I(a,b)-L(a,b)$ from Prop.~\ref{\en{EN}satzLI} and obtain:}%
\de{Schließlich setzen wir noch $L(b,a)=I(a,b)-L(a,b)$ aus Satz \ref{\en{EN}satzLI} ein und erhalten}\belowdisplayskip=-12pt
\begin{align*}
    2 c_0^2 \cdot (I(a,b)-L(a,b)) &= (c_0^2+S)\cdot I(a,b)\\
    \Longrightarrow\quad 2\cdot c_0^2\cdot L(a,b) &= (c_0^2-S)\cdot I(a,b)
\end{align*}

\end{proof}
\begin{thm}\label{\en{EN}satzgamma}
\en{The Gamma function $\Gamma(x) := \int_0^\infty t^{x-1}\cdot e^{-t} \mathrm{d}t$ satisfies for $\re(x)>0$:}%
\de{Für die Gamma-Funktion $\Gamma(x) := \int_0^\infty t^{x-1}\cdot e^{-t} \mathrm{d}t$ gilt für $\re(x)>0$:}
\begin{align*}
    \Gamma(x+1) &= x\cdot \Gamma(x)\qquad\text{\en{and}\de{und}}\qquad
    \Gamma\left(\frac{1}{2}\right) = \sqrt{\pi}
\end{align*}
\end{thm}
\begin{proof}
\en{We prove the functional equation integrating by parts:}%
\de{Die Funktionalgleichung folgt durch partielle Integration:}
\begin{align*}
    \Gamma(x+1) = \int\limits_0^\infty t^x\cdot e^{-t} \mathrm{d}t &= - \int\limits_0^\infty x\cdot t^{x-1}\cdot \left(-e^{-t}\right) \mathrm{d}t + \left[t^x\cdot\left(-e^{-t}\right) \right]_0^\infty = x\cdot \Gamma(x) 
\end{align*}
\en{To calculate $\Gamma(1/2)$ we substitute $s=\sqrt{t}$ and obtain $t=s^2$ and $\frac{\mathrm{d}t}{\mathrm{d}s}=2s$:}%
\de{Für $\Gamma(1/2)$ substituieren wir zunächst $s=\sqrt{t}$ und erhalten $t=s^2$ und $\frac{\mathrm{d}t}{\mathrm{d}s}=2s$:}
\begin{align*}
    \Gamma\left(\frac{1}{2}\right) &= \int\limits_0^\infty t^{-1/2}\cdot e^{-t} \mathrm{d}t
    =\int\limits_0^\infty s^{-1}\cdot e^{-s^2}\cdot 2s \mathrm{d}s = \int\limits_{-\infty}^\infty e^{-s^2} \mathrm{d}s
\end{align*}
\en{We square this integral to obtain a twodimensional integral:}%
\de{Dieses Integral quadrieren wir, um auf ein zweidimensionales Integral zu kommen:}
\begin{align*}
    \left(\Gamma\left(\frac{1}{2}\right)\right)^2
    = \left(\int\limits_{-\infty}^\infty e^{-x^2}\mathrm{d}x\right)\cdot\left(\int\limits_{-\infty}^\infty e^{-y^2}\mathrm{d}y\right)
    = \int\limits_{-\infty}^\infty\int\limits_{-\infty}^\infty e^{-(x^2+y^2)} \mathrm{d}x\mathrm{d}y
\end{align*}
\en{Now we use polar coordinates $x = r\cdot \cos \Phi$ and $y = r\cdot \sin \Phi$:}%
\de{Hier bietet sich ein Übergang zu Polarkoordinaten an, also $x = r\cdot \cos \Phi$ und $y = r\cdot \sin \Phi$:}
\begin{align*}
    \left(\Gamma\left(\frac{1}{2}\right)\right)^2 &= \int\limits_0^{\infty}\int\limits_0^{2\pi} e^{-r^2}\cdot r\mathrm{d}\Phi\mathrm{d}r = \int\limits_0^{\infty} e^{-r^2}\cdot 2\pi r\mathrm{d}r
    = \left[-\pi\cdot e^{-r^2}\right]_0^\infty = \pi
\end{align*}
\en{Thus the value of $\Gamma(1/2) = \sqrt{\pi}$ is also proven.}%
\de{Also ist auch der Wert $\Gamma(1/2) = \sqrt{\pi}$ bewiesen.}
\end{proof}

\begin{thm}\label{\en{EN}gammabeta}
\en{The Beta function $B(u,v) := \int_0^1 t^{u-1}\cdot(1-t)^{v-1} \mathrm{d}t$ satisfies for $\re(u)>0$ and $\re(v)>0$:}%
\de{Für die Betafunktion $B(x,y) := \int_0^1 t^{x-1}\cdot(1-t)^{y-1} \mathrm{d}t$ gilt im Bereich $\re(u)>0$ und $\re(v)>0$:}
\begin{align*}
    B(u,v) &= \frac{\Gamma(u)\cdot\Gamma(v)}{\Gamma(u+v)}
\end{align*}
\end{thm}
\begin{proof}
\en{We start with $\Gamma(u)\cdot\Gamma(v)$:}%
\de{Wir beginnen mit $\Gamma(u)\cdot\Gamma(v)$:}
\begin{align*}
  \Gamma(u)\cdot\Gamma(v) &= \int\limits_0^\infty t^{u-1}\cdot e^{-t} \mathrm{d}t \cdot \int\limits_0^\infty s^{v-1}\cdot e^{-s} \mathrm{d}s
  = \int\limits_0^\infty\int\limits_0^\infty t^{u-1}s^{v-1}\cdot e^{-t-s} \mathrm{d}t \mathrm{d}s
\end{align*}
\en{The substitution of $\left\{\begin{aligned}s &= x\cdot (1-y)\\t &= x\cdot y\end{aligned}\right\}$ 
or $\left\{\begin{aligned}x &= s+t\\y&=\frac{t}{x}=\frac{t}{s+t}\end{aligned}\right\}$ 
yields $0 < x < \infty$ and $0<y<1$.
The Jacobian matrix of this substitution is}%
\de{Mit der Substitution $\left\{\begin{aligned}s &= x\cdot (1-y)\\t &= x\cdot y\end{aligned}\right\}$ 
bzw. $\left\{\begin{aligned}x &= s+t\\y&=\frac{t}{x}=\frac{t}{s+t}\end{aligned}\right\}$ 
gilt $0 < x < \infty$ und $0<y<1$.
Die Jacobi-Matrix der Substitution ist}
\begin{align*}
    J&=\left(\begin{aligned}\frac{\mathrm{d}s}{\mathrm{d}x} && \frac{\mathrm{d}s}{\mathrm{d}y}\\[0.5ex] \frac{\mathrm{d}t}{\mathrm{d}x} && \frac{\mathrm{d}t}{\mathrm{d}y}\end{aligned}\right)
    =\left(\begin{aligned}1-y && -x\\[0.5ex]
    y && x\end{aligned}\right)\quad\Longrightarrow\quad \det(J) = x(1-y)+xy=x
\end{align*}
\en{This yields:}\de{Also gilt:}
\begin{align*}
   \Gamma(u)\cdot\Gamma(v) &= \int\limits_{x=0}^\infty\int\limits_{y=0}^1 (xy)^{u-1}\cdot (x(1-y))^{v-1}\cdot e^{-x} \cdot x\mathrm{d}y \mathrm{d}x\\
    &=\int\limits_0^\infty x^{u-1+v-1+1}\cdot e^{-x} \mathrm{d}x\cdot \int\limits_0^1 y^{u-1}(1-y)^{v-1}\mathrm{d}y= \Gamma(u+v)\cdot B(u,v)
\end{align*}
\en{A final division by $\Gamma(u+v)$ proves the equation of Prop.~\ref{\en{EN}gammabeta}.}%
\de{Eine abschließende Division durch $\Gamma(u+v)$ liefert die zu beweisende Gleichung.}
\end{proof}
\begin{thm}\label{\en{EN}satzwert}
\en{It holds}\de{Es gilt}
\begin{align*}
    L(\sqrt{2},1)\cdot I(\sqrt{2},1) &= \frac{\pi}{4}
\end{align*}
\end{thm}
\begin{proof}
\en{First, $\sin^2\Phi=1-\cos^2\Phi$ yields:}%
\de{Zunächst gilt wegen $\sin^2\Phi=1-\cos^2\Phi$:}
\begin{align*}
L(\sqrt{2},1) &= \int\limits_0^{\pi/2}\frac{\cos^2(\Phi)\mathrm{d}\Phi}{\sqrt{2\cos^2(\Phi)+\sin^2(\Phi)}}
= \int\limits_0^{\pi/2}\frac{\cos^2(\Phi)\mathrm{d}\Phi}{\sqrt{1+\cos^2(\Phi)}}\\
I(\sqrt{2},1) &= \int\limits_0^{\pi/2}\frac{\mathrm{d}\Phi}{\sqrt{2\cos^2(\Phi)+\sin^2(\Phi)}}
= \int\limits_0^{\pi/2}\frac{\mathrm{d}\Phi}{\sqrt{1+\cos^2(\Phi)}}
\end{align*}
\en{Then we substitute $x=\cos\Phi$ with $\frac{\mathrm{d}x}{\mathrm{d}\Phi} = -\sin\Phi=-\sqrt{1-\cos^2\Phi}=-\sqrt{1-x^2}$ and thus $\frac{\mathrm{d}\Phi}{\mathrm{d}x}=\frac{-1}{\sqrt{1-x^2}}$:}%
\de{Dann substituieren wir $x=\cos\Phi$, wobei $\frac{\mathrm{d}x}{\mathrm{d}\Phi} = -\sin\Phi=-\sqrt{1-\cos^2\Phi}=-\sqrt{1-x^2}$ und somit $\frac{\mathrm{d}\Phi}{\mathrm{d}x}=\frac{-1}{\sqrt{1-x^2}}$ gilt:}
\begin{align*}
L(\sqrt{2},1) &=  \int\limits_0^{\pi/2}\frac{\cos^2(\Phi)\mathrm{d}\Phi}{\sqrt{1+\cos^2(\Phi)}}
 = \int\limits_1^0\frac{x^2}{\sqrt{1+x^2}}\cdot\frac{-\mathrm{d}x}{\sqrt{1-x^2}}
 =\int\limits_0^1\frac{x^2\mathrm{d}x}{\sqrt{1-x^4}}\\
I(\sqrt{2},1) &=  \int\limits_0^{\pi/2}\frac{\mathrm{d}\Phi}{\sqrt{1+\cos^2(\Phi)}}
 = \int\limits_1^0\frac{1}{\sqrt{1+x^2}}\cdot\frac{-\mathrm{d}x}{\sqrt{1-x^2}}
 =\int\limits_0^1\frac{\mathrm{d}x}{\sqrt{1-x^4}}
\end{align*}
\en{Next we substitute $x=t^{1/4}$ with $\frac{\mathrm{d}x}{\mathrm{d}t}= \frac{1}{4}\cdot t^{-3/4}$ to obtain the Beta function $B(x,y)$ from Prop.~\ref{\en{EN}gammabeta}:}%
\de{Schließlich substituieren wir $x=t^{1/4}$ mit $\frac{\mathrm{d}x}{\mathrm{d}t}= \frac{1}{4}\cdot t^{-3/4}$, um auf die Betafunktion $B(x,y)$ aus Satz \ref{\en{EN}gammabeta} zu kommen:}
\begin{align*}
L(\sqrt{2},1) &=\int\limits_0^1\frac{t^{1/2}}{\sqrt{1-t}}\cdot\frac{\mathrm{d}t}{4\cdot t^{3/4}}
    =\int\limits_0^1\frac{1}{4}\cdot t^{-1/4}\cdot(1-t)^{-1/2}\mathrm{d}t 
    = \frac{1}{4}\cdot B\lk\frac{3}{4},\frac{1}{2}\rk\\
I(\sqrt{2},1) &=\int\limits_0^1\frac{1}{\sqrt{1-t}}\cdot\frac{\mathrm{d}t}{4\cdot t^{3/4}}
    =\int\limits_0^1\frac{1}{4}\cdot t^{-3/4}\cdot(1-t)^{-1/2}\mathrm{d}t 
    = \frac{1}{4}\cdot B\lk\frac{1}{4},\frac{1}{2}\rk
\end{align*}
\en{Now we replace the Beta function by Gamma functions as in Prop.~\ref{\en{EN}gammabeta} and use the properties of the Gamma function from Prop.~\ref{\en{EN}satzgamma}:}%
\de{Jetzt ersetzen wir die Betafunktionen durch Gammafunktionen mit Satz \ref{\en{EN}gammabeta} und verwenden dann die Eigenschaften der Gammafunktion aus Satz \ref{\en{EN}satzgamma}:}\belowdisplayskip=-8pt
\begin{align*}
L(\sqrt{2},1)\cdot I(\sqrt{2},1)
  &= \frac{1}{4}\cdot B\lk\frac{3}{4},\frac{1}{2}\rk\cdot\frac{1}{4}\cdot B\lk\frac{1}{4},\frac{1}{2}\rk\\
  &= \frac{1}{16}\cdot\frac{\Gamma(3/4)\cdot\Gamma(1/2)}{\Gamma(5/4)}\cdot\frac{\Gamma(1/4)\cdot\Gamma(1/2)}{\Gamma(3/4)}\\
  &= \frac{1}{16}\cdot\frac{\Gamma(3/4)\cdot\Gamma(1/2)}{1/4\cdot\Gamma(1/4)}\cdot\frac{\Gamma(1/4)\cdot\Gamma(1/2)}{\Gamma(3/4)}\\
  &= \frac{1}{4}\cdot \Gamma(1/2)^2 = \frac{1}{4}\cdot (\sqrt{\pi})^2 = \frac{\pi}{4}
\end{align*}
\end{proof}

\begin{proof}[\en{Proof of}\de{Beweis des} Thm.~\ref{\en{EN}tgauss}]\label{\en{EN}proofgaussi}
\en{The integrals $I(a,b)$ and $L(a,b)$ satisfy:}%
\de{Zunächst gilt für $I(a,b)$ und ebenso für $L(a,b)$:}
$$I(\lambda\cdot a,\lambda\cdot b) = \int\limits_0^{\pi/2}\frac{\mathrm{d}\Phi}{\sqrt{\lambda^2a^2\cos^2(\Phi)+\lambda^2b^2\sin^2(\Phi)}}=\frac{1}{|\lambda|}\cdot I(a,b)$$
\en{This implies $I(1,1/\sqrt{2})=\sqrt{2}\cdot I(\sqrt{2},1)$ and $L(1,1/\sqrt{2})=\sqrt{2}\cdot L(\sqrt{2},1)$ and then with Prop.~\ref{\en{EN}satzwert}:}%
\de{Hieraus folgt $I(1,1/\sqrt{2})=\sqrt{2}\cdot I(\sqrt{2},1)$ und $L(1,1/\sqrt{2})=\sqrt{2}\cdot L(\sqrt{2},1)$ und dann mit Satz \ref{\en{EN}satzwert}:}
$$L(1,1/\sqrt{2})\cdot I(1,1/\sqrt{2}) = \sqrt{2}^2\cdot L(\sqrt{2},1)\cdot I(\sqrt{2},1) = 2\cdot \frac{\pi}{4} = \frac{\pi}{2}$$
\en{Prop.~\ref{\en{EN}sumthm} yields}\de{Satz \ref{\en{EN}sumthm} liefert}
\begin{align*}
    2\cdot c_0^2\cdot L(a,b)\cdot I(a,b) &= (c_0^2-S)\cdot I(a,b)^2
\end{align*}
\en{Here we use $a=1$ and $b=1/\sqrt{2}$, thus $c_0^2=a^2-b^2=1/2$:}%
\de{Hier setzen wir $a=1$ und $b=1/\sqrt{2}$ ein, also $c_0^2=a^2-b^2=1/2$:}
\begin{align*}
    2\cdot \frac 1 2 \cdot \frac \pi 2 &= \left(\frac 1 2 - S\right) \cdot I(1,1/\sqrt{2})^2
\end{align*}
\en{And now Prop.~\ref{\en{EN}satziagm} proves:}%
\de{Nun verwenden wir Satz \ref{\en{EN}satziagm}:}
\begin{align*}
    \frac \pi 2 &= \left(\frac 1 2 - S\right) \cdot \left(\frac{\pi}{2\AGM(1,1/\sqrt{2})}\right)^2
\end{align*}
\en{Solving this for $\pi$ proves the Gaussian formula:}%
\de{Wenn wir diese Formel nach $\pi$ auflösen, erhalten wir die Gauß'sche Formel:}
$$\pi = \frac{4\cdot \AGM(1;1/\sqrt{2})^2}{1-2\cdot\sum_{j=1}^\infty 2^j\cdot c_j^2}$$
\en{The numerator of $p_N$ is $(a_N+b_N)^2=4\cdot a_{N+1}^2$ which converges to $4\cdot \AGM(1;1/\sqrt{2})^2$ (cf.~Prop.~\ref{\en{EN}Thm1}), thus the sequence $p_N$ converges to $\pi$.}%
\de{Der Zähler der Folge $p_N$ ist $(a_N+b_N)^2=4\cdot a_{N+1}^2$ und konvergiert (Satz \ref{\en{EN}Thm1}) gegen\linebreak $4\cdot \AGM(1;1/\sqrt{2})^2$, also konvergiert die Folge $p_N$ gegen $\pi$.}
\end{proof}

\begin{bem}
\en{Now we have proven that the Brent-Salamin algorithm approximates $\pi$.
We have already proven the equivalence of the three algorithms, thus the two algorithms by the Borwein brothers also approximate $\pi$.}%
\de{Nun haben wir bewiesen, dass der Brent-Salamin-Algorithmus gegen $\pi$ konvergiert.
Wir haben bereits die Äquivalenz der drei Algorithmen bewiesen, also folgt dass auch die beiden Algorithmen der Borwein-Brüder gegen $\pi$ konvergieren.}
\end{bem}

\vfill\pagebreak
\section{\en{Proof of Brent-Salamin's Quadratic Convergence}\de{Beweis der quadratischen Konvergenz}}
\label{\en{EN}kapkonv}
\renewcommand{\leftmark}{\en{Proof of Brent-Salamin's Quadratic Convergence}\de{Beweis der quadratischen Konvergenz von Brent-Salamin}}
\en{This chapter does not use the fact that the Brent-Salamin sequence $p_n$ converges to $\pi\approx 3\ko14159$. We only use the monotonic convergence of $a_n$ and $b_n$ from Prop.~\ref{\en{EN}Thm1} and use the symbol $\pi$ as a placeholder for the limit of this sequence $p_n$.}%
\de{Dieses Kapitel setzt nicht voraus, dass die Brent-Salamin-Folge gegen $\pi\approx 3\ko14159$ konvergiert. Wir setzen nur die monotone Konvergenz von $a_n$ und $b_n$ aus Satz \ref{\en{EN}Thm1} voraus und verwenden $\pi$ als Platzhalter für den Grenzwert von $p_n$.}\\
\begin{theo}\label{\en{EN}theo2}
\en{The sequence $\displaystyle p_n:=\frac{(a_n+b_n)^2}{1-2\cdot\sum_{j=1}^n 2^j c_j^2}$ of the Brent-Salamin algorithm converges quadratically to its limit which we denote with the symbol $\pi$:}%
\de{Die Folge $\displaystyle p_n:=\frac{(a_n+b_n)^2}{1-2\cdot\sum_{j=1}^n 2^j c_j^2}$ des Brent-Salamin-Algorithmus konvergiert quadratisch gegen ihren Grenzwert, den wir mit $\pi$ bezeichnen:}
$$|\pi-p_{n+1}| < 0\ko075 \cdot |\pi-p_{n}|^2$$
\en{In particular, the number of significant digits is approximately doubled with each iteration, but one has to do all calculations with the desired accuracy.}%
\de{Insbesondere wird die Anzahl gültiger Stellen mit jeder Iteration ungefähr verdoppelt, wobei man von Anfang an mit der gewünschten Zielgenauigkeit rechnen muss.}
\end{theo}
\begin{proof}
\noindent \en{First we denote the (de)nominator of the Brent-Salamin sequence with $X_n$ and $Y_n$.
In Prop.~\ref{\en{EN}Thm1} we proved that $a_{n}\rightarrow\AGM(1,1/\sqrt{2})$, thus it holds:}%
\de{Zunächst benennen wir Zähler und Nenner der Brent-Salamin-Folge mit $X_n$ und $Y_n$.
In Satz \ref{\en{EN}Thm1} haben wir bewiesen, dass $a_{n}\rightarrow\AGM(1,1/\sqrt{2})$ gilt, also folgt:}
\begin{align*}
    X_n &:= (a_n+b_n)^2=4\cdot a_{n+1}^2 &\longrightarrow&& X&:=4\cdot\AGM^2(1,1/\sqrt{2})\\
    Y_n &:= 1-2\cdot\sum_{j=1}^n 2^j\cdot c_j^2 &\longrightarrow&& Y&:=1-2\cdot\sum_{j=1}^\infty 2^j\cdot c_j^2\\
    p_n&:=X_n/Y_n&\longrightarrow&& \pi&:=X/Y
\end{align*}
\en{Now we denote the differences with}%
\de{Wir bezeichnen jetzt die Abweichungen mit}
\begin{align*}
    \varepsilon_n&:=X_n-X=4\cdot a_{n+1}^2-4\AGM^2(1,1/\sqrt{2})\\
    \delta_n&:=Y_n-Y = 2\cdot\sum_{j=n+1}^\infty 2^j\cdot c_j^2
\end{align*}
\en{In Prop.~\ref{\en{EN}Thm1} we proved that $a_{n+1}>\AGM(a,b)>b_{n+1}$ holds for all $n$. This yields:}%
\de{Weil für alle $n$ gilt $a_{n+1}>\AGM(a,b)>b_{n+1}$ (vgl. Satz \ref{\en{EN}Thm1}) folgt}
\begin{align*}
    0 < \varepsilon_n = 4\cdot a_{n+1}^2-4\AGM^2(1,1/\sqrt{2}) < 4\cdot a_{n+1}^2-4\cdot b_{n+1}^2= 4\cdot c_{n+1}^2
\end{align*}
\en{Since the summation of $\delta_n$ contains only positive terms, it holds $\delta_n>2\cdot2^{n+1}\cdot c_{n+1}^2$.
Finally we have proven in Prop.~\ref{\en{EN}Thm1} that $c_{n+1}^2<c_n^2/4$, thus we can use the geometric series to estimate $\delta_n$ (by setting the summation index to $j=n+1+k$):}%
\de{Weiter werden bei $\delta_n$ nur positive Zahlen summiert, also gilt $\delta_n>2\cdot2^{n+1}\cdot c_{n+1}^2$.
Schließlich ist $c_{n+1}^2<c_n^2/4$ (vgl. Satz \ref{\en{EN}Thm1}), also kann die Summe in $\delta_n$ mit der geometrischen Reihe abgeschätzt werden (setze hierfür den Summationsindex $j=n+1+k$):}
\begin{align*}
    \delta_n&= 2\cdot\sum_{k=0}^\infty 2^{n+1+k}\cdot c_{n+1+k}^2< 2\cdot\sum_{k=0}^\infty 2^{n+1+k}\cdot 4^{-k}\cdot c_{n+1}^2\\
    &=2^{n+2}\cdot c_{n+1}^2 \cdot\sum_{k=0}^\infty \left(\frac{1}{2}\right)^k = 2^{n+2}\cdot c_{n+1}^2 \cdot \frac{1}{1-\frac{1}{2}}
    = 2^{n+3}\cdot c_{n+1}^2
\end{align*}
\en{Thus we have proven:}%
\de{Insgesamt haben wir also bewiesen:}
\begin{align}\label{\en{EN}sechseins}
    0 < \varepsilon_n < 4\cdot c_{n+1}^2 \leq 
    2^{n+2}\cdot c_{n+1}^2 < \delta_n < 2^{n+3}\cdot c_{n+1}^2
\end{align}
\en{The difference between $p_n$ and its limit $\pi$ is:}%
\de{Für die Abweichung zwischen dem Folgenglied $p_n$ und dem Grenzwert $\pi$ gilt:}
\begin{align*}
    \pi-p_n &:= \frac{X}{Y}-\frac{X_n}{Y_n}
    = \left(\frac{X}{Y}-\frac{X}{Y_n}\right) + \left(\frac{X}{Y_n}-\frac{X_n}{Y_n}\right)\\
   &= \left(\frac{X}{Y}-\frac{X}{Y+\delta_n}\right) + \left(\frac{X}{Y_n}-\frac{X+\varepsilon_n}{Y_n}\right)\\
    &= \left(\frac{X\cdot(Y+\delta_n)-X\cdot Y}{Y\cdot(Y+\delta_n)}\right) + \left(\frac{-\varepsilon_n}{Y_n}\right)\\
    &= \left(\frac{X\cdot\delta_n}{Y\cdot(Y+\delta_n)}\right) - \frac{\varepsilon_n}{Y_n} = \frac{\pi\cdot\delta_n}{Y_n}- \frac{\varepsilon_n}{Y_n}
\end{align*}
\en{Here we use $1>Y_n>Y>0$ and $\delta_n>\varepsilon_n>0$ to obtain:}%
\de{Weiter ist $1>Y_n>Y>0$ und $\delta_n>\varepsilon_n>0$, also gilt:}
\begin{align*}
\frac{\pi\cdot\delta_n}{Y_n}- \frac{\varepsilon_n}{Y_n} &=  |\pi-p_n| <  \frac{\pi\cdot\delta_n}{Y_n}\\
\Longrightarrow\quad\frac{\pi\cdot\delta_n}{1}- \frac{\varepsilon_n}{Y} &<  |\pi-p_n| <  \frac{\pi\cdot\delta_n}{Y}
\end{align*}
\en{Now, eq.~(\ref{\en{EN}sechseins}) yields:}%
\de{Dann folgt mit (\ref{\en{EN}sechseins}):}
\begin{align}
    \frac{\pi\cdot2^{n+2}\cdot c_{n+1}^2}{1} - \frac{4\cdot c_{n+1}^2}{Y} &< |\pi-p_n| < \frac{\pi\cdot 2^{n+3}\cdot c_{n+1}^2}{Y}\nonumber\\
    \left(\pi\cdot2^{n+2} - \frac{\pi}{\AGM^2}\right)\cdot c_{n+1}^2 &< |\pi-p_n| < \left(\frac{\pi^2}{\AGM^2}\cdot 2^{n+1}\right)\cdot c_{n+1}^2\label{\en{EN}sechszwei}
\end{align}
\en{In the proof of Prop.~\ref{\en{EN}Thm1} we proved $c_{n+1}^2 =(a_n-b_n)^2/4$. Thus it holds:}%
\de{Im Beweis von Satz \ref{\en{EN}Thm1} haben wir $c_{n+1}^2 =(a_n-b_n)^2/4$ bewiesen. Hieraus folgt:}
\begin{align*}
    c_{n+1}^2 &= \frac{(a_n-b_n)^2}{4} = \frac{(a_n-b_n)^2\cdot(a_n+b_n)^2}{4\cdot(a_n+b_n)^2} 
    = \frac{(a_n^2-b_n^2)^2}{16\cdot\left(\frac{a_n+b_n}{2}\right)^2} = \frac{c_n^4}{16\cdot a_{n+1}^2}
\end{align*}
\en{Here we see the quadratic convergence:}%
\de{In dieser Zeile erkennen wir die quadratische Konvergenz:}
\begin{align}\label{\en{EN}dreidrei}
    c_{n+1}^2 = \frac{(c_n^2)^2}{16\cdot a_{n+1}^2} < \frac{c_n^4}{16\cdot \AGM^2}
\end{align}
\en{Using first (\ref{\en{EN}sechszwei}) and then (\ref{\en{EN}dreidrei}) yields:}%
\de{Hieraus folgt mit (\ref{\en{EN}sechszwei}):}
\begin{align}
    \frac{|\pi-p_{n+1}|}{|\pi-p_{n}|^2} &< \frac{\left(\frac{\pi^2}{\AGM^2}\cdot 2^{n+2}\right)\cdot c_{n+2}^2}{\left(\pi\cdot2^{n+2} - \frac{\pi}{\AGM^2}\right)^2\cdot c_{n+1}^4}
    <\frac{\left(\frac{\pi^2}{\AGM^2}\cdot 2^{n+2}\right)\cdot \frac{c_{n+1}^4}{16\cdot \AGM^2}}{\left(\pi\cdot2^{n+2} - \frac{\pi}{\AGM^2}\right)^2\cdot c_{n+1}^4}\nonumber\\
    &= \frac{\left(\frac{1}{\AGM^2}\cdot 2^{n+2}\right)\cdot \frac{1}{16\cdot \AGM^2}}{\left(2^{n+2} - \frac{1}{\AGM^2}\right)^2}
    = \frac{2^{n-2}}{\left(2^{n+2}\cdot\AGM^2-1\right)^2}
\end{align}
\en{Here we use $\AGM>b_1$ (Prop.~\ref{\en{EN}Thm1}) and $\AGM(1,1/\sqrt{2})>\sqrt{1\cdot 1/\sqrt{2}}=\sqrt[4]{1/2}$ and $2^n\geq 1$ to obtain:}%
\de{Hier nutzen wir $\AGM>b_1$ (Satz \ref{\en{EN}Thm1}) bzw. $\AGM(1,1/\sqrt{2})>\sqrt{1\cdot 1/\sqrt{2}}=\sqrt[4]{1/2}$ und $2^n\geq 1$ und erhalten:}
\begin{align*}
    \frac{|\pi-p_{n+1}|}{|\pi-p_{n}|^2} &< \frac{2^{n-2}}{\left(2^{n+2}\cdot\AGM^2-1\right)^2} < \frac{2^{n-2}}{\left(2^{n+2}\cdot\sqrt{1/2}-2^n\right)^2}\\
    &= \frac{2^{-n}}{4\cdot\left(4\cdot\sqrt{1/2}-1\right)^2} < 0\ko075\cdot 2^{-n} < 0\ko075
\end{align*}
\en{This proves the quadratic convergence of $p_n$.}%
\de{Somit ist die quadratische Konvergenz von $p_n$ gegen $\pi$ bewiesen.}
\end{proof}

\begin{bem}
\en{Using $p_1>3\ko14057$, Thm.~\ref{\en{EN}theo2} yields $|\pi-p_{48}|<10^{-5\ko7\cdot10^{14}}$.
Thus $p_{48}=\widehat{\pi}_{48}=\pi_{24}$ is closer to $\pi$ than the current world record of $3\cdot 10^{14}$ digits (May 2025).}%
\de{Mit $p_1>3\ko14057$ folgt aus Thm.~\ref{\en{EN}theo2}: $|\pi-p_{48}|<10^{-5\ko7\cdot10^{14}}$.
Insbesondere liegt $p_{48}=\widehat{\pi}_{48}=\pi_{24}$ näher an $\pi$ als der aktuelle Rekord von $3\cdot 10^{14}$ Dezimalen (Mai 2025).}
\end{bem}

\begin{bem}
\en{For $p_{48}$, it even holds $|\pi-p_{48}|<10^{-1\ko7\cdot10^{15}}$. This follows from the stronger error bound $0 < \pi-p_n < \left(2^{n+4}\pi^2-8\pi\right)\cdot\exp\lk-2^{n+1}\pi\rk$ which is proven in \cite[eq.~(20)]{Brent2018}. }%
\de{Tatsächlich gilt sogar $|\pi-p_{48}|<10^{-1\ko7\cdot10^{15}}$, was aus der in \cite[Glg.~(20)]{Brent2018} bewiesenen Fehlerabschätzung von $0 < \pi-p_n < \left(2^{n+4}\pi^2-8\pi\right)\cdot\exp\lk-2^{n+1}\pi\rk$ folgt.}
\end{bem}

\vfill
\hrule
\fancyhead[OL]{\emph{\en{References}\de{Literatur}}}
\label{\en{EN}literat}\bibliographystyle{plain}
{\raggedright
\renewcommand{\refname}{\en{References}\de{Literatur}}

}

  \vfill\pagebreak
  \invisiblepart{Deutsche Version}
  \enfalse

\end{document}